\documentclass[leqno,english]{amsart}
\usepackage[T1]{fontenc}
\usepackage[latin9]{inputenc}
\usepackage{color}
\usepackage{babel}
\usepackage{enumitem}
\usepackage{amstext}
\usepackage{amsthm}
\usepackage{amssymb}
\usepackage[unicode=true,pdfusetitle,
 bookmarks=true,bookmarksnumbered=true,bookmarksopen=false,
 breaklinks=false,pdfborder={0 0 0},pdfborderstyle={},backref=page,colorlinks=true]
 {hyperref}
\hypersetup{
 citecolor=blue,linkcolor=red}

\makeatletter
\numberwithin{equation}{section}
\numberwithin{figure}{section}
\numberwithin{table}{section}
\theoremstyle{plain}
\newtheorem{thm}{\protect\theoremname}[section]
\theoremstyle{plain}
\newtheorem{conjecture}[thm]{\protect\conjecturename}
\theoremstyle{plain}
\newtheorem{prop}[thm]{\protect\propositionname}
\theoremstyle{definition}
\newtheorem{defn}[thm]{\protect\definitionname}
\theoremstyle{plain}
\newtheorem{cor}[thm]{\protect\corollaryname}
\theoremstyle{plain}
\newtheorem{lem}[thm]{\protect\lemmaname}
\theoremstyle{remark}
\newtheorem{notation}[thm]{\protect\notationname}

\usepackage{tikz,fancyhdr,enumitem}
\usetikzlibrary{scopes,calc,intersections,through,backgrounds,arrows,decorations.pathmorphing,positioning,fit,petri,decorations.markings,matrix}
\usepackage{hyperref,bookmark,dsfont,bbm}
\usepackage[all]{xy}
\hypersetup{linktocpage}

\subjclass[2020]{Primary 05B35; Secondary 52B40, 14N20}

\makeatother

\providecommand{\conjecturename}{Conjecture}
\providecommand{\corollaryname}{Corollary}
\providecommand{\definitionname}{Definition}
\providecommand{\lemmaname}{Lemma}
\providecommand{\notationname}{Notation}
\providecommand{\propositionname}{Proposition}
\providecommand{\theoremname}{Theorem}

\begin{document}
\title[Realizable Sticky Matroid Conjecture]{Realizable Sticky Matroid Conjecture}
\author{Jaeho Shin}
\address{Department of Mathematical Sciences, Seoul National University, Gwanak-ro
1, Gwanak-gu, Seoul 08826, South Korea}
\email{j.shin@snu.ac.kr}
\keywords{Sticky matroid conjecture, Kantor's conjecture, modular extension,
hypermodular matroids, hyperplane arrangements}
\begin{abstract}
We give a criterion for modular extension of rank-$4$ hypermodular
matroids, and prove a weakening of Kantor's conjecture for rank-$4$
realizable matroids. This proves the sticky matroid conjecture and
Kantor's conjecture for realizable matroids due to an argument of
Bachem, Kern, and Bonin, and due to an equivalence argument of Hochstättler
and Wilhelmi, respectively.
\end{abstract}

\maketitle
\tableofcontents{}

\section{Introduction}

We work on finite matroids throughout the paper. Given matroids, there
are several ways to produce a new matroid from them. Surgery techniques
for a rank-$k$ matroid $M$ on $S$ include restriction, deletion,
contraction, truncation, dualizing, etc. One can also pullback or
pushforward any matroid via a set-theoretic map, \cite[Section 1]{j-hope}.
Taking into account the well-known one-to-one correspondence between
matroids and (matroid) base polytopes, cutting the hypersimplex $\Delta_{S}^{k}$
with a half-space of the form $\left\{ x(A)\le\rho\right\} $ for
a positive integer $\rho$ produces a matroid, whether new or not,
\cite[Lemma 5.1]{j-hope2}.

For a number of matroids $M_{1},\dots,M_{n}$ with positive ranks,
we can obtain a new matroid by finding a common extension. For instance,
their direct sum $M_{1}\oplus\cdots\oplus M_{n}$ and the matroid
union $M_{1}\vee\cdots\vee M_{n}$ are matroids, but the rank increase
is inevitable. 

To obtain a matroid of the same rank, one may take the naive union
of the base collections of $M_{1},\dots,M_{n}$ and see if the union
is a base collection of a matroid. But, if $\mathrm{BP}_{M_{1}},\dots,\mathrm{BP}_{M_{n}}$
are face-fitting and glue to a convex polytope, then the union is
indeed a base collection of a matroid of the same rank, \cite[Lemma 3.15]{j-hope}.

Basically, producing a new matroid from given ones is an extension
problem, which is hard just as for other extension problems. Poljak
and Turzik considered a matroid extension problem, \cite{Sticky}.
If $N$ is a matroid with the ground set $E(N)=E_{1}\cup E_{2}$,
they called $N$ an amalgam of two matroids $N|_{E_{1}}$ and $N|_{E_{2}}$.
Conversely, fix a matroid $M$ and let $N_{1}$ and $N_{2}$ be two
arbitrary matroids such that $E(M)=E(N_{1})\cap E(N_{2})$ and $M=N_{1}|_{E(M)}=N_{2}|_{E(M)}$.
If there always exists an amalgam of $N_{1}$ and $N_{2}$, the matroid
$M$ is called a \textbf{sticky matroid}. They showed a modular matroid
is a sticky matroid, but its converse has been a conjecture: 
\begin{conjecture}[Sticky Matroid Conjecture]
 Every sticky matroid is modular.
\end{conjecture}

A matroid is called \textbf{hypermodular} if every pair of two distinct
maximal proper flats is modular. Bonin proposed the following weakening
of Kantor's conjecture which we will call the hypermodular matroid
conjecture, cf. \cite{Kantor,Bonin}.
\begin{conjecture}[Hypermodular Matroid Conjecture]
 Any hypermodular matroid of sufficiently large rank is a restriction
of a modular matroid of the same rank.
\end{conjecture}

This conjecture holds in rank $3$, Proposition \ref{prop:rk3-mod}.
Bonin pointed out that if it holds in rank $4$, the sticky matroid
conjecture holds due to an argument of Bachem, Kern, and Bonin, \cite{BK88,Bonin}.
Also, Hochstättler and Wilhelmi \cite{HW19} showed that the sticky
matroid conjecture is equivalent to Kantor's conjecture which is:
\begin{conjecture}[Kantor's Conjecture]
 Every hypermodular matroid is embeddable in a modular matroid.
\end{conjecture}

We prove the hypermodular matroid conjecture for rank-$4$ realizable
matroids, and hence prove both the sticky matroid conjecture and Kantor's
conjecture for realizable matroids.

\subsection*{Acknowledgements}

The author would like to thank Joseph Bonin that he kindly informed
him of the recent related works. The author is also grateful to James
Oxley for his comments on Conjecture 15.9.4 of his book \cite{Oxley}.
He would like to thank Michael Wilhelmi for helpful conversations.

\section{Preliminaries}

Among many different definitions of a matroid, we intensively use
rank axioms and flat axioms.
\begin{prop}[Matroid Rank Axioms]
For a finite set $S$, let $r$ be a $\mathbb{Z}_{\ge0}$-valued
function defined on the power set $2^{S}$ of $S$ such that 
\begin{enumerate}[label=(R\arabic*),itemsep=1pt]
\item \label{enu:(R1)}$0\le r(A)\le\left|A\right|$ for all $A\in2^{S}$,
\item \label{enu:(R2)}$r(A)\le r(B)$ for all $A,B\in2^{S}$ with $A\subseteq B$,
\item \label{enu:(R3)}$r(A\cup B)+r(A\cap B)\le r(A)+r(B)$ for all $A,B\in2^{S}$,
\end{enumerate}
where \ref{enu:(R3)} is called the \textbf{submodularity}. Then,
$r$ is the rank function of a matroid.
\end{prop}

\begin{prop}[Matroid Flat Axioms]
 A nonempty subcollection $\mathcal{A}\subseteq2^{S}$ with $S\in\mathcal{A}$
is the lattice of a matroid on $S$ if it satisfies the following
axioms.
\begin{enumerate}[label=(F\arabic*),itemsep=1pt]
\item \label{enu:(F1)} For $\ensuremath{F,L\in\mathcal{A}}$, one has
$F\cap L\in\mathcal{A}$.
\item \label{enu:(F2)} For $F\in\mathcal{A}$ and $s\in S-F$, the smallest
member $L$ of $\mathcal{A}$ containing $F\cup\left\{ s\right\} $
covers $F$, that is, there is no member of $\mathcal{A}$ between
$F$ and $L$.
\end{enumerate}
Let $M$ be the corresponding matroid. We write $\mathcal{A}=\mathcal{L}(M)$
and denote by $\mathcal{L}^{\left(i\right)}(M)$ the collection of
rank-$i$ flats of $M$.
\end{prop}

\begin{defn}
If $r(A\cup B)+r(A\cap B)=r(A)+r(B)$ in \ref{enu:(R3)}, the pair
$\left\{ A,B\right\} $ is called \textbf{modular}. The difference
$r(A)+r(B)-r(A\cup B)-r(A\cap B)\ge0$ is called the \textbf{modular
defect} of the pair.
\end{defn}

Throughout the paper, a modular defect of a matroid $M$ means a modular
defect of a pair of flats of $M$ unless otherwise specified. 
\begin{defn}
A flat $F$ is called \textbf{modular} if $\left\{ F,L\right\} $
is a modular pair for all flats $L$. The matroid $M$ is called \textbf{modular}
if every flat is modular.
\end{defn}

\begin{prop}
\label{prop:mod-rk1}Let $F$ be a rank-$1$ flat of a matroid $M$.
For an arbitrary flat $L$, the pair $\left\{ F,L\right\} $ is modular,
and hence the flat $F$ is modular.
\end{prop}

\begin{proof}
Since either $F\subset L$ or $F\cap L=\emptyset$, either $r(F\cap L)=r(F)$
and $r(F\cup L)=r(L)$, or $r(F\cap L)=0$ and $r(F\cup L)=r(L)+1$.
In either case, $r(F\cap L)+r(F\cup L)=r(F)+r(L)$ and $\left\{ F,L\right\} $
is a modular pair.
\end{proof}
\begin{defn}
Let $M$ be a matroid of rank $k\ge3$. If every pair of two corank-$1$
flats is a modular pair, we say that $M$ is \textbf{hypermodular}.
\end{defn}

For rank-$3$ matroids, the hypermodularity is equivalent to the modularity.
\begin{prop}
\label{prop:rk3-mod}Every rank-$3$ hypermodular matroid is modular.
\end{prop}

\begin{proof}
By the hypermodularity, every pair of two rank-$2$ flats is a modular
pair. Applying Proposition \ref{prop:mod-rk1} finishes the proof.
\end{proof}
\begin{defn}
A matroid $M$ is called \textbf{inseparable} or connected if it has
no nontrivial separator, and \textbf{separable} or disconnected otherwise.
A subset $A$ of the ground set $E(M)$ is called inseparable or separable
if $M|_{A}$ is.
\end{defn}

\begin{defn}
For any matroid $M$, we denote by $\kappa(M)$ the number of connected
components of a matroid $M$. Fix a matroid $M$. For a subset $F$
of $E(M)$, we say that $F$ is a \textbf{non-degenerate} subset if
$\kappa(M|_{F})+\kappa(M/F)=\kappa(M)+1$.
\end{defn}

\begin{prop}[{\cite[Lemma 4.21]{j-hope}}]
\label{prop:Shi19}  Let $M$ be a rank-$3$ inseparable matroid.
If $F$ and $L$ are two rank-$2$ flats with $F\cap L\neq\emptyset$
and $F\cup L=E(M)$, then $F\cap L$ is a degenerate flat, and vice
versa. These flats are unique.
\end{prop}

\begin{cor}
\label{cor:Degen}The matroid of Proposition \ref{prop:Shi19} is
not hypermodular.
\end{cor}

\begin{proof}
There are rank-$2$ subsets $\left\{ f_{1},f_{2}\right\} \subset F-L$
and $\left\{ l_{1},l_{2}\right\} \subset L-F$. Then, $\left\{ f_{1},l_{1}\right\} $
and $\left\{ f_{2},l_{2}\right\} $ are rank-$2$ flats which have
empty intersection. Therefore, the matroid is not hypermodular.
\end{proof}

\section{\label{sec:Hypermodular}Hypermodular Matroids}

We introduce two interesting properties of hypermodular matroids.
\begin{prop}
\label{prop:inherit-HM}Let $M$ be a hypermodular matroid of rank
$\ge3$, and $A$ a flat of corank $\ge3$. Then, the contraction
$M/A$ is also a hypermodular matroid.
\end{prop}

\begin{proof}
Let $F$ and $L$ be any two distinct corank-$1$ flats of $M$ that
contain $A$, then $F-A$ and $L-A$ are two distinct corank-$1$
flats of $M/A$, and vice versa. The pair $\left\{ F-A,L-A\right\} $
is a modular pair of $M/A$ by the following:
\begin{align*}
r_{M/A}(F-A)+r_{M/A}(L-A) & =r(F)-r(A)+r(L)-r(A)\\
 & =r(F\cup L)-r(A)+r(F\cap L)-r(A)\\
 & =r_{M/A}(F\cup L-A)+r_{M/A}(F\cap L-A).
\end{align*}
Hence, $M/A$ is a hypermodular matroid.
\end{proof}
\begin{prop}
\label{prop:Properness-FUL}Let $M$ be a loopless  hypermodular matroid
of rank $\ge3$. For any two distinct corank-$1$ flats $F$ and $L$,
either their union $F\cup L$ is a proper subset of $E(M)$ or $F\cup L=E(M)$
and one of them is a union of two corank-$2$ flats where the intersection
of these two flats is nonempty if $r(M)\ge4$.
\end{prop}

\begin{proof}
Let $A$ be a corank-$3$ flat that is contained in $F\cap L$. The
contraction $M/A$ is a loopless rank-$3$ hypermodular matroid by
Proposition \ref{prop:inherit-HM}. Further, $M/A$ is inseparable
if and only if $F\cup L\neq E(M)$ by Corollary \ref{cor:Degen}.
If $M/A$ is separable, $F-A$ and $L-A$ are rank-$2$ flats of $M/A$
whose union is $E(M)-A$. Moreover, one of them is separable and is
a union of two distinct rank-$1$ flats of $M/A$. So, one of $F$
and $L$ is a union of two distinct corank-$2$ flats of $M$. If
$r(M)\ge4$, then $r(A)\ge1$ and those corank-$2$ flats have a nonempty
intersection.
\end{proof}
Henceforth, we focus on rank-$4$ hypermodular matroids and investigate
those who are not modular.
\begin{prop}
\label{prop:equiv}Let $M$ be a loopless rank-$4$ hypermodular matroid.
There are no two disjoint flats of rank $3$ and $2$, respectively,
if and only if no rank-$3$ flat contains two disjoint rank-$2$ flats.
\end{prop}

\begin{proof}
To prove the only if direction, let $T$ be a rank-$3$ flat containing
two disjoint rank-$2$ flats $L$ and $A$. For any $e\in E(M)-T\neq\emptyset$,
the flat $F=\overline{A\cup\left\{ e\right\} }$ is a rank-$3$ flat
whose intersection with $T$ is $A$ by the hypermodularity. Thus,
$F$ and $L$ are two disjoint flats of rank $3$ and $2$, respectively.

To prove the if direction, let $F$ and $L$ be two disjoint flats
of rank $3$ and $2$, respectively. For any $f\in F$, the flat $T=\overline{L\cup\left\{ f\right\} }$
is a rank-$3$ flat whose intersection with $F$, say $A=T\cap F$,
is a rank-$2$ flat by the hypermodularity, where these rank-$2$
flats $L$ and $A$ are disjoint and contained in a rank-$3$ flat
$T$.
\end{proof}
\begin{lem}
\label{lem:equiv}Let $M$ be a loopless rank-$4$ hypermodular matroid.
There is no pair of two disjoint flats of rank $3$ and $2$, respectively,
if and only if $M$ is modular.
\end{lem}

\begin{proof}
Suppose that $M$ is modular. For any two flats $F$ and $L$ we have
$r(F\cap L)=r(F)+r(L)-r(F\cup L)$. If $r(F)=3$ and $r(L)=2$, then
$r(F\cap L)=1$ and $F\cap L\neq\emptyset$.\smallskip{}

Conversely, suppose there are no two disjoint flats of rank $3$ and
$2$, respectively. We check all the cases and prove $M$ is a modular
matroid. By the hypermodularity, every pair of two rank-$3$ flats
is a modular pair.

Let $F$ and $L$ be flats of rank $3$ and $2$, respectively. By
assumption, $F\cap L\neq\emptyset$, and therefore either $L\subset F$
or $L\nsubseteq F$ with $r(F\cap L)=1$ where in the latter case,
$r(F\cap L)+r(F\cup L)=5=r(F)+r(L)$. Thus, $\left\{ F,L\right\} $
is a modular pair.

Let $L$ and $A$ be two distinct rank-$2$ flats with $L\cap A\neq\emptyset$,
then $r(L\cap A)=1$ and $r(L\cup A)\ge3$. By the submodularity,
$4\le r(L\cap A)+r(L\cup A)\le r(L)+r(A)=4$ and equality holds. Therefore,
$\left\{ L,A\right\} $ is a modular pair.

Let $L$ and $A$ be two disjoint rank-$2$ flats. Then, $\overline{L\cup A}$
is a flat of rank $>3$ by Proposition \ref{prop:equiv}. Therefore,
$\overline{L\cup A}=E(M)$, and $\left\{ L,A\right\} $ is a modular
pair.

Let $T$ be a rank-$1$ flat and $A$ an arbitrary flat, then $\left\{ T,A\right\} $
is a modular pair by Proposition \ref{prop:mod-rk1}.

Thus, $M$ is a modular matroid, and the proof is done.
\end{proof}
Rank-$4$ hypermodular but non-modular matroids have interesting properties.
We begin with the following lemma which plays a crucial role in the
later sections.
\begin{lem}
\label{lem:Disjoint0} Let $M$ be a loopless hypermodular but non-modular
matroid of rank $4$. Then, $M$ is inseparable. Let $F$ and $L$
be disjoint flats of rank $3$ and $2$, respectively, and $A_{1},\dots,A_{n}$
all those rank-$3$ flats containing $L$. Then, ${\textstyle E(M)=A_{1}\cup\cdots\cup A_{n}}$
with $n\ge3$ where all $A_{i}-L$ are pairwise disjoint and all $A_{i}-F\sqcup L$
are nonempty. Moreover, all $A_{i}$ and $F$ are inseparable.
\end{lem}

\begin{proof}
Note that for any $e\in E(M)-L$, the flat $\overline{L\cup\left\{ e\right\} }$
is a rank-$3$ flat, and every rank-$3$ flat containing $L$ arises
this way. Take any $f_{1}\in F$, then $A_{1}:=\overline{L\cup\left\{ f_{1}\right\} }$
is a rank-$3$ flat and $T_{1}:=A_{1}\cap F$ is a rank-$2$ flat
by the hypermodularity of $M$. For each $i=1,2,\dots$ we take $f_{i+1}\in F-T_{1}\cup\cdots\cup T_{i}$
unless $F=T_{1}\cup\cdots\cup T_{i}$. Then, $A_{i+1}:=\overline{L\cup\left\{ f_{i+1}\right\} }$
is a rank-$3$ flat and $T_{i+1}:=A_{i+1}\cap F$ is a rank-$2$ flat.
Since $F$ is a finite set, this process terminates and $F=T_{1}\cup\cdots\cup T_{n}$
for some $n\ge2$. For two distinct $i,j\in\left[n\right]$, we have
$\left(A_{i}-L\right)\cap\left(A_{j}-L\right)=\emptyset$, and $F=T_{1}\sqcup\cdots\sqcup T_{n}$.
So, $F$ is an inseparable flat. 

Take an arbitrary $e\in E(M)$. By the hypermodularity of $M$, the
rank-$3$ flat $\overline{L\cup\left\{ e\right\} }$ intersects $F$
and $T_{i}$ for some $i$. Then, for any $e_{i}\in\overline{L\cup\left\{ e\right\} }\cap T_{i}\neq\emptyset$,
both $\overline{L\cup\left\{ e\right\} }$ and $A_{i}$ are rank-$3$
flats containing the rank-$3$ subset $L\cup\left\{ e\right\} $,
and so they are the same flat, which means $e\in A_{i}$. Therefore,
$E(M)=A_{1}\cup\cdots\cup A_{n}$, and at least two of $A_{i}-F\sqcup L=A_{i}-T_{i}\sqcup L$
are nonempty since otherwise $E(M)=F\sqcup A_{i}$ for some $i\in\left[n\right]$,
but neither $F$ nor $A_{i}$ is a union of two rank-$2$ flats whose
intersection is nonempty, a contradiction to Proposition \ref{prop:Properness-FUL}.
Similarly, we have $n\ge3$.

Further, all $A_{i}-T_{i}\sqcup L$ are nonempty. To show this, suppose
that $A_{1}\neq T_{1}\sqcup L$, $A_{2}\neq T_{2}\sqcup L$ and $A_{3}=T_{3}\sqcup L$
without loss of generality. Take $a\in A_{2}-T_{2}\sqcup L\neq\emptyset$,
then $\overline{T_{1}\sqcup\left\{ a\right\} }$ is a rank-$3$ flat
such that $\overline{T_{1}\sqcup\left\{ a\right\} }\cap T_{3}\subseteq\overline{T_{1}\sqcup\left\{ a\right\} }\cap\left(F-T_{1}\right)=\emptyset$
and $\overline{T_{1}\sqcup\left\{ a\right\} }\cap L\subseteq\overline{T_{1}\sqcup\left\{ a\right\} }\cap\left(A_{1}-T_{1}\right)=\emptyset$,
and hence $\overline{T_{1}\sqcup\left\{ a\right\} }\cap A_{3}=\emptyset$,
which contradicts the hypermodularity of $M$. Thus, $A_{i}\neq T_{i}\sqcup L$
and $A_{i}-F\sqcup L\neq\emptyset$ for all $i\in\left[n\right]$,
and in particular, $A_{i}$ are inseparable flats.

Since $E(M)=L\sqcup\left(\bigsqcup_{i\in\left[n\right]}\left(A_{i}\cap F\right)\right)\sqcup\left(\bigsqcup_{i\in\left[n\right]}\left(A_{i}-F\sqcup L\right)\right)$
with $n\ge3$, there is no nontrivial separator of $M$, and $M$
is an inseparable matroid.
\end{proof}
\begin{prop}
\label{prop:Insep-rk3}Let $M$ be an inseparable rank-$4$ matroid
that is hypermodular. Then, every rank-$3$ flat is inseparable.
\end{prop}

\begin{proof}
Let $F$ be a separable rank-$3$ flat, then $F=F_{1}\sqcup F_{2}$
with two flats $F_{1}$ and $F_{2}$ of rank $1$ and $2$, respectively.
Take any $a\in E(M)-F\neq\emptyset$, then $\overline{F_{2}\sqcup\left\{ a\right\} }$
is a rank-$3$ flat. Since $M$ is inseparable, there is $b\in E(M)-\overline{F_{2}\sqcup\left\{ a\right\} }\sqcup F_{1}\neq\emptyset$.
Then, by the submodularity of $M$, 
\begin{align*}
4=r(F_{2})+r(\overline{\{a,b\}}) & \ge r(F_{2}\cup\overline{\{a,b\}})+r(F_{2}\cap\overline{\{a,b\}})\\
 & =4+r(F_{2}\cap\overline{\{a,b\}}).
\end{align*}
 Thus, $r(F_{2}\cap\overline{\{a,b\}})=0$ and $F_{2}\cap\overline{\{a,b\}}=\emptyset$
since $M$ is loopless. If $M$ is modular, $F\cap\overline{\{a,b\}}\neq\emptyset$.
If $M$ is not modular, since $M$ is hypermodular and the rank-$3$
flat $F$ is separable, no rank-$2$ flat is disjoint from $F$ by
Lemmas \ref{lem:equiv} and \ref{lem:Disjoint0}, and so $F\cap\overline{\{a,b\}}\neq\emptyset$.
Then, $\emptyset\neq F\cap\overline{\{a,b\}}=F_{1}\cap\overline{\{a,b\}}=F_{1}$
and $\overline{\{a,b\}}=\overline{F_{1}\sqcup\{a\}}$. This implies
that $\overline{F_{1}\sqcup\{x\}}$ for all $x\in E(M)-F$ are the
same rank-$2$ flat. Then, we have $E(M)-F\subseteq\overline{F_{1}\sqcup\{a\}}$
and $E(M)=F_{2}\sqcup\overline{F_{1}\sqcup\{a\}}$, which contradicts
that $M$ is inseparable. Therefore, every rank-$3$ flat is inseparable.
\end{proof}

\section{\label{sec:Sticky}Extension of Rank-$4$ Hypermodular Matroids}

In this section, we give a criterion for single-element extension
of rank-$4$ hypermodular matroids that decreases the total sum of
modular defects, Theorem \ref{thm:criterion}. Prior to that, we need
Lemmas \ref{lem:Span-1} and \ref{lem:Span-2}.
\begin{notation}
Let $M$ be a matroid. For any flat $J$ and any subcollection $\mathcal{A}$
of the geometric lattice of $M$, we denote $J\vee\mathcal{A}=\left\{ \overline{J\cup A}:A\in\mathcal{A}\right\} $
and for $k\in\mathbb{Z}_{\ge0}$, 
\[
\left(J\vee\mathcal{A}\right)^{\left(k\right)}=\left\{ \overline{J\cup A}:A\in\mathcal{A},r(J\cup A)=k\right\} .
\]
\end{notation}

\begin{lem}
\label{lem:Span-1}Assume the setting of Lemma \ref{lem:Disjoint0}
with $T_{i}=A_{i}\cap F$ for all $i\in\left[n\right]$ and $T_{0}=L$,
and let $\mathcal{T}:=\left\{ T_{0},\dots,T_{n}\right\} $. Let $\mathcal{J}$
be the collection of the rank-$2$ flats $J\subset E(M)-F\sqcup L$
such that $\left(J\vee\mathcal{T}\right)^{\left(3\right)}$ has cardinality
at least $2$, and $\mathcal{J}_{+}:=\mathcal{J}\cup\mathcal{T}$.
Let $\mathcal{J}^{\sharp}$ be the collection of the rank-$3$ flats
containing some $T\in\mathcal{T}$. Then, 
\[
\mathcal{J}^{\sharp}\subseteq\left\{ \overline{J\cup J'}:J,J'\in\mathcal{J}_{+},J\neq J'\right\} .
\]
\end{lem}

\begin{proof}
Let $X\in\mathcal{J}^{\sharp}$, then $X$ contains some $T\in\mathcal{T}$.
If $T=L$, then $X=A_{i}$ for some $i\in\left[n\right]$ by Lemma
\ref{lem:Disjoint0} where $A_{i}=\overline{T_{i}\cup L}$. Then,
since $F=\overline{T_{1}\cup T_{2}}$, we may assume $X\notin\left\{ F,A_{1},\dots,A_{n}\right\} $.
Else if $T\neq L$, then $T=T_{i}$ for some $i\in\left[n\right]$,
and $X\cap A_{j}\neq T_{i}$ for any $j\neq i$ is a rank-$2$ flat
by the hypermodularity of $M$. So, $X\cap A_{j}\in\mathcal{J}$ and
$X=\overline{T\cup\left(X\cap A_{j}\right)}$. Thus, $\mathcal{J}^{\sharp}\subseteq\left\{ \overline{J\cup J'}:J,J'\in\mathcal{J}_{+},J\neq J'\right\} $.
\end{proof}
\begin{lem}
\label{lem:Span-2}Assume the setting of Lemma \ref{lem:Span-1}.
Suppose that 
\[
\mathcal{J}^{\sharp}\supseteq\left\{ \overline{J\cup J'}:J,J'\in\mathcal{J}_{+},J\neq J'\right\} .
\]
\begin{enumerate}[leftmargin=3em,itemsep=2pt]
\item \label{enu:J-sharp}Then, $\mathcal{J}^{\sharp}=\left\{ \overline{J\cup J'}:J,J'\in\mathcal{J}_{+},J\neq J'\right\} $.
\item \label{enu:J-disjoint}Every two distinct elements of $\mathcal{J}_{+}$
are disjoint.
\item \label{enu:J-partition}$\mathcal{J}_{+}$ is a partition of $E(M)$.
\item \label{enu:J-UB}For any $X\in\mathcal{L}^{\left(2\right)}(M)-\mathcal{J}_{+}$,
the cardinality of $\left(X\vee\mathcal{T}\right)^{\left(3\right)}$
is $1$.
\item \label{enu:J-LB}For any $X\in\mathcal{J}^{\sharp}$ and $J\in\mathcal{J}_{+}$,
either $J\subset X$ or $J\cap X=\emptyset$.
\end{enumerate}
\end{lem}

\begin{proof}
(1) Lemma \ref{lem:Span-1} tells that $\mathcal{J}^{\sharp}=\left\{ \overline{J\cup J'}:J,J'\in\mathcal{J}_{+},J\neq J'\right\} $.

(2) Let $J$ and $J'$ be two distinct elements of $\mathcal{J}$.
By (\ref{enu:J-sharp}) and Lemma \ref{lem:Disjoint0}, we have $\overline{J\cup L}=A_{i}$
and $\overline{J'\cup L}=A_{j}$ for some $i,j\in\left[n\right]$.
If $i\neq j$, then $J\cap J'\subseteq A_{i}\cap A_{j}=L$. If $i=j$,
take $k\in\left[n\right]-\left\{ i\right\} $, then $J\cap J'\subseteq\left(\overline{T_{k}\cup J}\right)\cap\left(\overline{T_{k}\cup J'}\right)=T_{k}$.
In either case, $J\cap J'=\emptyset$. Therefore, any two distinct
elements of $\mathcal{J}_{+}$ are disjoint by Lemma \ref{lem:Disjoint0}.

(3) Take any $x\in E(M)-F\sqcup L$. Then, $J=\overline{T_{1}\cup\left\{ x\right\} }\cap\overline{T_{2}\cup\left\{ x\right\} }$
is a rank-$2$ flat containing $x$ where $J\in\mathcal{J}$. Therefore,
$\mathcal{J}_{+}$ is a partition of $E(M)$ by (\ref{enu:J-disjoint}).

(4) Take any rank-$2$ flat $X\in\mathcal{L}^{\left(2\right)}(M)-\mathcal{J}_{+}$,
then $\left|\left(X\vee\mathcal{T}\right)^{\left(3\right)}\right|\le1$.
Moreover, $X$ is written as the disjoint union of rank-$1$ flats
by (\ref{enu:J-partition}), say $X=Y_{1}\sqcup\cdots\sqcup Y_{\lambda}=\overline{Y_{1}\cup Y_{2}}$.
Take $J_{1},J_{2}\in\mathcal{J}_{+}$ such that $Y_{1}\subset J_{1}$
and $Y_{2}\subset J_{2}$. By (\ref{enu:J-sharp}), $X\subset\overline{J_{1}\cup J_{2}}\in\mathcal{J}^{\sharp}$
and $\overline{J_{1}\cup J_{2}}$ contains some $T\in\mathcal{T}$.
So, $\overline{J_{1}\cup J_{2}}=\overline{X\cup T}$ and $\left|\left(X\vee\mathcal{T}\right)^{\left(3\right)}\right|\ge1$.

(5) By (\ref{enu:J-sharp}) and (\ref{enu:J-partition}), every $X\in\mathcal{J}^{\sharp}$
is a disjoint union of elements of $\mathcal{J}_{+}$, which implies
that for any $J\in\mathcal{J}_{+}$, either $J\subset X$ or $J\cap X=\emptyset$.
\end{proof}
\begin{notation}
For a graded poset $\mathcal{F}$, we denote by $\mathcal{F}^{\left(r\right)}$
the $r$-th graded piece of $\mathcal{F}$, that is, the collection
of rank-$r$ elements of $\mathcal{F}$.
\end{notation}

\pagebreak{}
\begin{thm}
\label{thm:criterion}Assume the setting of Lemma \ref{lem:Span-1}.
The following are equivalent.
\begin{enumerate}[leftmargin=2.6em,itemsep=2pt]
\item \label{enu:Ext1}$\mathcal{J}^{\sharp}\supseteq\left\{ \overline{J\cup J'}:J,J'\in\mathcal{J}_{+},J\neq J'\right\} .$
\item \label{enu:Ext2}There is a rank-$4$ hypermodular matroid $\tilde{M}$
on $E(M)\cup\left\{ \tilde{m}\right\} $ with $M=\tilde{M}|_{E(M)}$
whose total sum of modular defects is less than that of $M$ and:
\begin{enumerate}[leftmargin=1.8em]
\item $\mathcal{L}^{\left(4\right)}(\tilde{M})=E(M)\cup\left\{ \tilde{m}\right\} $
and $\mathcal{L}^{\left(0\right)}(\tilde{M})=\left\{ \emptyset\right\} $,
\item $\mathcal{L}^{\left(3\right)}(\tilde{M})=\left(\mathcal{L}^{\left(3\right)}(M)-\mathcal{J}^{\sharp}\right)\cup\left\{ J\cup\left\{ \tilde{m}\right\} :J\in\mathcal{J}^{\sharp}\right\} $,
\item $\mathcal{L}^{\left(2\right)}(\tilde{M})=\left(\mathcal{L}^{\left(2\right)}(M)-\mathcal{J}_{+}\right)\cup\left\{ J\cup\left\{ \tilde{m}\right\} :J\in\mathcal{J}_{+}\right\} $,
\item $\mathcal{L}^{\left(1\right)}(\tilde{M})=\mathcal{L}^{\left(1\right)}(M)\cup\left\{ \tilde{m}\right\} $.
\end{enumerate}
\end{enumerate}
\end{thm}

\begin{proof}
We may assume that $E(M)=\left[m\right]$ without loss of generality.
\smallskip{}

Suppose (\ref{enu:Ext2}). Observe that $\mathcal{J}^{\sharp}=\mathcal{L}^{\left(2\right)}(\tilde{M}/\left\{ \tilde{m}\right\} )$
and $\mathcal{J}_{+}=\mathcal{L}^{\left(1\right)}(\tilde{M}/\left\{ \tilde{m}\right\} )$
where $\tilde{M}/\left\{ \tilde{m}\right\} $ is a rank-$3$ modular
matroid by Propositions \ref{prop:inherit-HM} and \ref{prop:rk3-mod}.
Let $J$ and $J'$ be any two distinct elements of $\mathcal{J}_{+}$.
Then, $X=\overline{J\cup J'}$ is a rank-$2$ flat of $\tilde{M}/\left\{ \tilde{m}\right\} $,
and $X\cup\left\{ \tilde{m}\right\} $ is a rank-$3$ flat of $\tilde{M}$.
Therefore, $X$ is a rank-$3$ flat of $M$. If $X=F$, then $X\in\mathcal{J}^{\sharp}$.
If $X\neq F$, let $Y=X\cap F$, then $Y$ is a rank-$1$ flat of
$\tilde{M}/\left\{ \tilde{m}\right\} $. Since $\mathcal{T}$ is the
collection of all rank-$1$ flats of $\tilde{M}/\left\{ \tilde{m}\right\} $
that are contained in $F$, cf. Lemma \ref{lem:Disjoint0}, we have
$Y\in\mathcal{T}$ and hence $X\in\mathcal{J}^{\sharp}$. Therefore,
$\left\{ \overline{J\cup J'}:J,J'\in\mathcal{J}_{+},J\neq J'\right\} \subseteq\mathcal{J}^{\sharp}$.\smallskip{}

Conversely, suppose (\ref{enu:Ext1}). Let $\tilde{m}=m+1$. We define
a rank-$4$ graded poset $\mathcal{F}$ of subsets of $\left[\tilde{m}\right]$
ordered by set inclusion, and then show it is a geometric lattice.
By Lemma \ref{lem:Span-2}(\ref{enu:J-sharp}), $\mathcal{J}^{\sharp}=\left\{ \overline{J\cup J'}:J,J'\in\mathcal{J}_{+},J\neq J'\right\} $.
Note that every element of $\mathcal{J}_{+}$ is an intersection of
two elements of $\mathcal{J}^{\sharp}$. Define a map $\epsilon:\mathcal{L}(M)\rightarrow2^{\left[\tilde{m}\right]}$
such that: 
\[
\epsilon(X)=\begin{cases}
X\cup\left\{ \tilde{m}\right\}  & \text{if }X\in\mathcal{J}^{\sharp}\cup\mathcal{J}_{+},\\
X & \text{otherwise}.
\end{cases}
\]
 Let $\mathcal{F}:=\epsilon(\mathcal{L}(M))\cup\left\{ \tilde{m}\right\} \subset2^{\left[\tilde{m}\right]}$
and define a grading on it as follows:
\begin{itemize}
\item $\mathcal{F}^{\left(4\right)}=E(M)\cup\left\{ \tilde{m}\right\} $
and $\mathcal{F}^{\left(0\right)}=\left\{ \emptyset\right\} $,
\item $\mathcal{F}^{\left(3\right)}=\left(\mathcal{L}^{\left(3\right)}(M)-\mathcal{J}^{\sharp}\right)\cup\left\{ J\cup\left\{ \tilde{m}\right\} :J\in\mathcal{J}^{\sharp}\right\} $,
\item $\mathcal{F}^{\left(2\right)}=\left(\mathcal{L}^{\left(2\right)}(M)-\mathcal{J}_{+}\right)\cup\left\{ J\cup\left\{ \tilde{m}\right\} :J\in\mathcal{J}_{+}\right\} $,
\item $\mathcal{F}^{\left(1\right)}=\mathcal{L}^{\left(1\right)}(M)\cup\left\{ \tilde{m}\right\} $.
\end{itemize}
Observe that the lattice structure of $\mathcal{F}$ restricted to
the elements of rank $\ge2$ is the same as that of $\mathcal{L}(M)$.
So, if $\mathcal{F}$ is a geometric lattice, its matroid is hypermodular.
To show $\mathcal{F}$ satisfies \ref{enu:(F1)}, it suffices to show
that for any $X\cup\left\{ \tilde{m}\right\} \in\mathcal{J}^{\sharp}\cup\left\{ \tilde{m}\right\} $
and $J\cup\left\{ \tilde{m}\right\} \in\mathcal{J}_{+}\cup\left\{ \tilde{m}\right\} $
with $X\in\mathcal{J}^{\sharp}$ and $J\in\mathcal{J}_{+}$, their
intersection $\left(X\cap J\right)\cup\left\{ \tilde{m}\right\} $
is an element of $\mathcal{F}$. Indeed, by Lemma \ref{lem:Span-2}(\ref{enu:J-LB}),
either $\left(X\cap J\right)\cup\left\{ \tilde{m}\right\} =J\cup\left\{ \tilde{m}\right\} \in\mathcal{F}$
or $\left(X\cap J\right)\cup\left\{ \tilde{m}\right\} =\left\{ \tilde{m}\right\} \in\mathcal{F}$.

To show $\mathcal{F}$ satisfies \ref{enu:(F2)}, it suffices to show
that for any $X\in\left(\mathcal{L}^{\left(2\right)}(M)-\mathcal{J}_{+}\right)\cup\mathcal{L}^{\left(1\right)}(M)$,
the smallest element $X^{\ast}$ of $\mathcal{F}$ that contains $X\cup\left\{ \tilde{m}\right\} $
covers $X=\epsilon(X)$. If $X\in\mathcal{L}^{\left(2\right)}(M)-\mathcal{J}_{+}$,
there is a unique element $X^{\ast}$ of $\mathcal{F}^{\left(3\right)}$
that contains $X\cup\left\{ \tilde{m}\right\} $ by Lemma \ref{lem:Span-2}(\ref{enu:J-UB}).
Then, since $X^{\ast}-\left\{ \tilde{m}\right\} $ covers $X$ in
$\mathcal{L}(M)$, $X^{\ast}$ also covers $X$ in $\mathcal{F}$.
If $X\in\mathcal{L}^{\left(1\right)}(M)$, by Lemma \ref{lem:Span-2}(\ref{enu:J-partition}),
there is a unique element $X^{\ast}$ of $\mathcal{F}^{\left(2\right)}$
that contains $X\cup\left\{ \tilde{m}\right\} $, which covers $X$.

Thus, $\mathcal{F}$ is a geometric lattice. Let $\tilde{M}$ be its
matroid, then $\tilde{M}$ is a hypermodular matroid on $E(\tilde{M})=E(M)\cup\left\{ \tilde{m}\right\} $
with $M=\tilde{M}|_{E(M)}$ satisfying all (a)(b)(c)(d). Observe that
the modular defect of a pair of two flats never increases while the
modular defect of $\left\{ \epsilon(F),\epsilon(L)\right\} $ is $0$
which is one less than that of $\left\{ F,L\right\} $. Therefore,
the total sum of modular defects of $\tilde{M}$ is less than that
of $M$.
\end{proof}

\section{\label{sec:Geometry}Geometric Interpretation and Main Theorem}

Throughout this section, we assume our matroids are loopless. We review
first basic notions of the theory of combinatorial hyperplane arrangements
of \cite{j-hope}.
\begin{defn}
\label{def:Hyperplanes-subspaces}Fix a loopless matroid $M$. A \textbf{subspace}
of $M$ is a matroid $M/F$ for a flat $F$. When considering $M/F$
as a subspace of $M$ we often write $\eta(M/F)$ instead. The collection
of subspaces of $M$ and the dual lattice of the geometric lattice
of $M$ are isomorphic as lattices. The \textbf{subspace} \textbf{dimension}
and\textbf{ codimension in} $M$ of the subspace $\eta(M/F)$ are
defined as: 
\[
\mathrm{sdim}\,\eta(M/F):=r(M/F)-1\quad\mbox{and}\quad\mathrm{scodim}_{M}\,\eta(M/F):=r(M|_{F}).
\]
A subspace is called a \textbf{point}, a \textbf{line}, and a \textbf{plane}
if its dimension is $0$, $1$, and $2$, respectively. It is called
a \textbf{hyperplane} if its codimension is $1$.
\end{defn}

\begin{defn}
\label{def:abstract_HA}For each $i\in E(M)$, the pair $B_{i}=(\eta(M/\overline{\left\{ i\right\} }),i)$
is called a \textbf{labeled hyperplane} with \textbf{label} $i$.
A \textbf{hyperplane arrangement} $\mathrm{HA}_{M}$ of $M$ on $E(M)$
is defined as the following pair: 
\[
\mathrm{HA}_{M}:=\left(M,\{B_{i}\}_{i\in E(M)}\right).
\]
\end{defn}

The loopless matroid $M$ is realizable if and only if there is a
projective realization of $\mathrm{HA}_{M}$ over a field. In particular,
if $M$ is a rank-$3$ loopless matroid, $\mathrm{HA}_{M}$ is a line
arrangement and the geometry of $\mathrm{HA}_{M}$ is the same as
that of lines in a projective plane: every two distinct lines meet
at a point and there exists a unique line passing through two distinct
points. In general, distinct subspaces $\eta(M/F_{1}),\dots,\eta(M/F_{m})$
meet at a point if and only if $r(F_{1}\cup\cdots\cup F_{m})=r(M)-1$,
which is called the Bézout's theorem for $\mathrm{HA}_{M}$.\smallskip{}

The hypermodularity of a loopless matroid $M$ is equivalent to the
property that every two distinct points of the hyperplane arrangement
$\mathrm{HA}_{M}$ is connected by a line of $\mathrm{HA}_{M}$. For
rank-$4$ loopless matroids, the modularity means the above property
and the property that every line and any point not on it are contained
in a plane.\smallskip{}

Proposition \ref{prop:equiv} and Lemma \ref{lem:equiv} are incorporated
into the following translation.
\begin{prop}
\label{prop:real-equiv}Let $M$ be a loopless rank-$4$ hypermodular
matroid. Then, $M$ is modular if and only if every pair of a point
and a line lie on a plane if and only if every two lines intersecting
at a point lie on a plane.
\end{prop}

From Lemma \ref{lem:Disjoint0}, we obtain the following.
\begin{lem}
\label{lem:real-disjoint}Let $M$ be a loopless rank-$4$ hypermodular
matroid that is not modular. Let $\eta(M/F)$ and $\eta(M/L)$ be
a point and a line, respectively, such that no plane contains both,
and let $\eta(M/A_{i})$, $i\in\left[n\right]$, be all the points
lying on the line $\eta(M/L)$. Then, every plane contains at least
one of those points $\eta(M/A_{i})$, $i\in\left[n\right]$.
\end{lem}

Now, we prove the main theorem of this paper.
\begin{thm}
\label{thm:main}Every realizable rank-$4$ hypermodular matroid is
a restriction of a realizable rank-$4$ modular matroid.
\end{thm}

\begin{proof}
Let $M$ be a realizable rank-$4$ hypermodular matroid that is not
modular. We may assume that $M$ is loopless. By Lemma \ref{lem:equiv},
there are two disjoint flats $F$ and $L$ of rank $3$ and $2$,
respectively. Assume the setting of Lemma \ref{lem:Span-1}. Since
$M$ is realizable, $\mathrm{HA}_{M}$ has a realization in the projective
space $\mathbb{P}^{3}$ over some field. For $X\in\mathcal{L}(M)$,
denote by $\sigma(M/X)$ the realization of $\eta(M/X)$ in $\mathbb{P}^{3}$.

Then, all lines $\sigma(M/T_{i})$, $i\in\left[n\right]\cup\left\{ 0\right\} $,
are contained in a plane of $\mathbb{P}^{3}$. Moreover, the line
$\sigma(M/J)$ with any $J\in\mathcal{J}$, which passes through two
distinct points on the plane, intersects all those lines $\sigma(M/T_{i})$,
$i\in\left[n\right]\cup\left\{ 0\right\} $. Thus, every line $\sigma(M/J)$
with $J\in\mathcal{J}_{+}$ intersects any other line $\sigma(M/J')$
with $J'\in\mathcal{J}_{+}-\left\{ J\right\} $. This implies that
$r(\overline{J\cup J'})=3$ for every two distinct $J,J'\in\mathcal{J}_{+}$.
Suppose $\overline{J\cup J'}\neq F$. Then, there is a unique line
$\ell$ in $\mathbb{P}^{3}$ that connects two distinct points $\sigma(M/F)$
and $\sigma(M/\overline{J\cup J'})$. By Lemma \ref{lem:real-disjoint},
$\ell$ passes through some point $\sigma(M/A_{i})$, and hence $\ell=\sigma(M/T_{i})$.
Therefore, $\overline{J\cup J'}$ contains $T_{i}\in\mathcal{T}$
and so $\overline{J\cup J'}\in\mathcal{J}^{\sharp}$. Thus, we showed:
\[
\mathcal{J}^{\sharp}\supseteq\left\{ \overline{J\cup J'}:J,J'\in\mathcal{J}_{+},J\neq J'\right\} .
\]
 By Theorem \ref{thm:criterion}, there is a rank-$4$ single-element
extension $\tilde{M}$ of $M$ whose total sum of modular defects
is less than that of $M$. Since $\tilde{M}$ is hypermodular and
realizable, we repeat this extension process until we obtain a rank-$4$
matroid such that either there is no pair of two disjoint flats of
rank $3$ and $2$, respectively, or the total sum of modular defects
is $0$. This realizable matroid is modular by Lemma \ref{lem:equiv}
and its restriction to $E(M)$ is $M$. The proof is complete.
\end{proof}
As a corollary, we obtain the following theorem.
\begin{thm}
Both the sticky matroid conjecture and Kantor's conjecture hold for
realizable matroids.
\end{thm}


\begin{thebibliography}{References}
\bibitem[BK88]{BK88} A. Bachem and W. Kern, \emph{On sticky matroids},
Discrete Math. \textbf{69} (1988), 11\textendash 18.

\bibitem[Bon11]{Bonin} J. Bonin, \emph{A note on the sticky matroid
conjecture}, Ann. Comb. \textbf{15} (2011), 619\textendash 624.

\bibitem[HW19]{HW19} W. Hochstättler and M. Wilhelmi, \emph{Sticky
matroids and Kantor\textquoteright s Conjecture}, Algebra Universalis
\textbf{80} (2019), no. 12.

\bibitem[Kan74]{Kantor} W. Kantor, \emph{Dimension and embedding
theorems for geometric lattices}, J. Combin. Theory Ser. A \textbf{17}
(1974), 173\textendash 195.

\bibitem[Oxl11]{Oxley} J. Oxley, \emph{Matroid Theory}, second edition,
Oxf. Grad. Texts Math., vol. \textbf{21}, Oxford Univ. Press, New
York, 2011.

\bibitem[PT82]{Sticky} S. Poljak and D. Turzik, \emph{A note on sticky
matroids}, Discrete Math. \textbf{42} (1982), 119\textendash 123.

\bibitem[Sch03]{Schrijver} A. Schrijver, \emph{Combinatorial Optimization:
Polyhedra and Efficiency}, Springer-Verlag, Berlin, 2003.

\bibitem[Shi19]{j-hope} J. Shin, \emph{Geometry of matroids and hyperplane
arrangements}, arXiv:1912.12449, 2019.

\bibitem[Shi20]{j-hope2} J. Shin, \emph{Biconvex polytopes and tropical
linearity}, arXiv:2002.11307, 2020.
\end{thebibliography}
\end{document}